\documentclass[11pt]{article}
\usepackage{amsmath, amssymb, amscd, amsthm, amsfonts}
\usepackage{graphicx}
\usepackage{hyperref}
\usepackage{amssymb}
\usepackage{amsmath}
\usepackage{amsthm}
\usepackage{amsfonts}

\usepackage{a4wide}
\usepackage{longtable}
\usepackage{multirow}

\usepackage[latin1]{inputenc}
\usepackage[english]{babel}
\usepackage[T1]{fontenc}
\usepackage{amssymb,amsmath,amsthm,amstext,amsfonts,latexsym}

\usepackage{pgf,tikz}
\usepackage{mathrsfs}
\usetikzlibrary{arrows}

\newtheorem{lemma}{Lemma}[section]
\newtheorem{proposition}{Proposition}[section]
\newtheorem{theorem}{Theorem}[section]
\newtheorem{corollary}{Corollary}[section]

\newtheorem{remark}{Remark}[section]
\theoremstyle{definition}
\newtheorem{definition}{Definition}[section]

\theoremstyle{remark}

\usepackage[all]{xy}
\title{On the capitulation problem of some pure metacyclic fields\\ of degree 20}

\date{}

\begin{document}

\maketitle
\begin{center}
{\sc Fouad ELMOUHIB }\\
{\footnotesize Department of Mathematics and Computer Sciences,\\
Mohammed first University, Oujda, Morocco,\\
Correspondence: fouad.cd@gmail.com}\\
\vspace{0.7cm}
{\sc Mohamed TALBI }\\
{\footnotesize Regional Center of Professions of Education and Training,\\
ksirat1971@gmail.com}\\

\vspace{0.7cm}
{\sc Abdelmalek AZIZI }\\
{\footnotesize Department of Mathematics and Computer Sciences,\\
Mohammed first University, Oujda, Morocco,\\
abdelmalekazizi@yahoo.fr}
\end{center}

\begin{abstract}
Let $\Gamma \,=\, \mathbb{Q}(\sqrt[5]{n})$ be a pure quintic field, where $n$ is a positive integer $5^{th}$ power-free, $k_0\,=\,\mathbb{Q}(\zeta_5)$ be the cyclotomic field containing  a primitive  $5^{th}$ root of unity $\zeta_5$, and $k\,=\,\mathbb{Q}(\sqrt[5]{n},\zeta_5)$ the normal closure of $\Gamma$. Let $k_5^{(1)}$ be the Hilbert $5$-class field of $k$, $C_{k,5}$ the $5$-ideal classes group of $k$, and $C_{k,5}^{(\sigma)}$ the group of ambiguous classes under the action of $Gal(k/k_0)$ = $\langle\sigma\rangle$. When $C_{k,5}$ is of type $(5,5)$ and rank $C_{k,5}^{(\sigma)}\,=\,1$, we study the capitulation problem of the $5$-ideal classes of $C_{k,5}$ in the six intermediate extensions of $k_5^{(1)}/k$.
\end{abstract}


\section{Introduction}\label{introduction}
Let $K$ be a number field of finite degree over $\mathbb{Q}$. Let $F$ be an unramified extension of $K$ of finite degree and let $\mathcal{O}_F$ be its ring of integers. We say that an ideal $\mathcal{A}$ ( or the ideal class of $\mathcal{A}$) of $K$ capitulates in $F$, if it becomes principal in $F$, i.e if $\mathcal{A}\mathcal{O}_F$ is principal in $F$. Let $p$ a prime number and $C_{K,p}$ be the $p$-ideal class group of $K$ such that $C_{K,p}$ is of type $(p,p)$. Let $K_p^{(1)}$ be the Hilbert $p$-class field of $K$, i.e the maximal abelian unramified extension of $K$. The most important result on capitulation is the "Principal Ideal Theorem"  conjectured by D. Hilbert in \ref{Hilbert} and proved by His student P. Furtw$\displaystyle\stackrel{\text{..}}{a}$ngler in \ref{Artin}. The theorem asserts that the class
group of a given number field capitulates completely in its Hilbert class field. Having established the "Principal Ideal Theorem", the question remains which ideal classes of K capitulate in a field which lies between $K$ and its Hilbert class field. A number of researchers have studied this question, Taussky and Scholz in \ref{Tauss and skolz} were the first who treated the capitulation in unramified cyclic degree-3-extensions of several imaginary quadratic fields. Tannaka and Terada proved that for a cyclic extension $K/K_0$ with Galois group $G$, the $G$-invariant ideal classes in $K$ capitulate in the genus field of $K/K_0$, for more details see \ref{Teradda}. Hilbert in his celebrated Zahlbericht, he proved his Theorem 94 which states that in an unramfied cyclic extension $L/K$, there are non-principal ideals which capitulates in $L$. For more several results on capitulation we refer the reader to   \ref{ayadi}, \ref{azizi}, \ref{ism}, \ref{Myaki}, \ref{Talbi} \ref{taussky}...\\
In this paper we consider $p = 5$ and $k\,=\,\mathbb{Q}(\sqrt[5]{n},\zeta_5)$, $(\zeta_5\,=\,e^{\frac{2i\pi}{5}})$, be the normal closure of the pure quintic field $\Gamma\,=\,\mathbb{Q}(\sqrt[5]{n})$, where $n$ is a positive integer $5^{th}$ power free. We suppose that the $5$-class group $C_{k,5}$ of $k$ is of type $(5,5)$, then the Hilbert $5$-class field $k_5^{(1)}$ of $k$ has degree $25$ over $k$, and the extension $k_5^{(1)}/k$ contains six intermediate extensions. By $k_0\,=\,\mathbb{Q}(\zeta_5)$ we denote the cyclotomic field containing a primitive $5^{th}$ root of unity $\zeta_5$. Let $C_{k,5}^{(\sigma
)}$ be the group of ambiguous ideal classes of $k$ under the action of $Gal(k/k_0)\,=\,\langle\sigma\rangle$. According to \ref{FOU1}, if $C_{k,5}$ is of type $(5,5)$ and rank $C_{k,5}^{(\sigma
)}\,=\,1$ we have three forms of the radicand $n$ as follows
\begin{itemize}
\item[-] $n\,=\,5^ep\,\not\equiv\,\pm1\pm7\,(\mathrm{mod}\,25)$ with $e\in\{1,2,3,4\}$ and p prime such that $p\,\not\equiv\,-1\,(\mathrm{mod}\,25)$.
\item[-] $n\,=\,p^eq\,\equiv\,\pm1\pm7\,(\mathrm{mod}\,25)$ with $e\in\{1,2,3,4\}$, p and q are primes such that $p\,\not\equiv\,-1\,(\mathrm{mod}\,25)$, $q\,\not\equiv\,\pm7\,(\mathrm{mod}\,25)$.
\item[-] $n\,=\,p^e\,\equiv\,\pm1\pm7\,(\mathrm{mod}\,25)$ with $e\in\{1,2,3,4\}$ and p prime such that $p\,\equiv\,-1\,(\mathrm{mod}\,25)$.
\end{itemize}

We will study the capitulation of $C_{k,5}$ in the six  intermediate extensions of $k_5^{(1)}/k$ In this cases.


\begin{center}

$\textbf{Notations. \ }$
\end{center}
Throughout this paper, we use the following notations:
\begin{itemize}
 
 \item $\Gamma\,=\,\mathbb{Q}(\sqrt[5]{n})$: a pure quintic field, where $n\neq 1$ is a $5^{th}$ power-free positive integer.
 
 \item $k_0\,=\,\mathbb{Q}(\zeta_5)$, the cyclotomic field, where $\zeta_5\,=\,e^{2i\pi/5}$ a primitive $5^{th}$ root of unity.
 
 \item $k\,=\,\mathbb{Q}(\sqrt[5]{n},\zeta_5)$: the normal closure of $\Gamma$, a quintic Kummer extension of $k_0$.
 
\item $\Gamma^{'},\,\Gamma^{''},\,\Gamma^{'''},\,\Gamma^{''''},\,$ the four conjugates quintic fields of $\Gamma$, contained in $k$. 
 
 \item $\langle\tau\rangle\,=\,\operatorname{Gal}(k/\Gamma)$ such that $\tau^4\,=\,id,\, \tau^3(\zeta_5)\,=\,\zeta_5^3,\, \tau^2(\zeta_5)\,=\,\zeta_5^4,\, \tau(\zeta_5)\,=\,\zeta_5^2$ and $\tau(\sqrt[5]{n})\,=\,\sqrt[5]{n}$.
 
 \item $\langle\sigma\rangle\,=\,\operatorname{Gal}(k/k_0)$ such that $\sigma^5\,=\,id,\ \sigma(\zeta_5)\,=\,\zeta_5$ and $\sigma(\sqrt[5]{n})\,=\,\zeta_5\sqrt[5]{n},\, \sigma^2(\sqrt[5]{n})\,=\,\zeta_5^2\sqrt[5]{n},\,\\\\ \sigma^3(\sqrt[5]{n})\,=\,\zeta_5^3\sqrt[5]{n},\, \sigma^4(\sqrt[5]{n})\,=\,\zeta_5^4\sqrt[5]{n} $.
 
 \item For a number field $L$, denote by:
\begin{itemize}
   \item $\mathcal{O}_{L}$: the ring of integers of $L$;
   \item $C_{L}$, $h_{L}$, $C_{L,5}$: the class group, class number, and $5$-class group of $L$.
   \item $L_5^{(1)}, L^*$: the Hilbert $5$-class field of $L$, and the absolute genus field of $L$. 
  \end{itemize}
  
\end{itemize}

\begin{center}
\begin{tikzpicture}
\begin{scope}[xscale=2,yscale=2]
  \node (P) at (0,5) {$\mathbf{k_5^{(1)}}$};
  \node (J) at (0.5,4) {$\mathbf{K_4}$};
  \node (K) at (1,4) {$\mathbf{K_5}$};
  \node (L) at (1.5,4) {$\mathbf{K_6}$};
  \node (M) at (-0.5,4) {$\mathbf{K_3}$};
  \node (N) at (-1,4) {$\mathbf{K_2}$};
  \node (O) at (-1.5,4) {$\mathbf{K_1}$};
  
  \node (A) at (0,3) {$\mathbf{k}$};
  \node (B) at (1,2){$\mathbf{\Gamma}$ };
  \node (C) at (1.5,2){$\mathbf{\Gamma'}$ };
  \node (D) at (2,2){$\mathbf{\Gamma''}$ };
  \node (E) at (2.5,2){$\mathbf{\Gamma'''}$ };
  \node (F) at (3,2){$\mathbf{\Gamma''''}$ };  
  \node (G) at (-1,1) {$\mathbf{k_0}$};
  \node (H)   at (0,0) {$\mathbb{Q}$ };
  \node (I)   at (0,-0.5) {\underline{Figure 1}};
  \draw [-,>=latex] (G) -- (A)   node[midway,above,rotate=40] {};
  \draw [-,>=latex] (B) -- (A)   node[midway,above,rotate=-40] {};
  \draw [-,>=latex] (C) -- (A)   node[midway,below left] {};
  \draw [-,>=latex] (D) -- (A)   node[midway,below left] {};
  \draw [-,>=latex] (E) -- (A)   node[midway,below left] {};
  \draw [-,>=latex] (F) -- (A)   node[midway,below left] {};
  \draw [-,>=latex] (H) -- (G)   node[midway,below right] {};
  \draw [-,>=latex] (H) -- (B)   node[midway,below right] {};
  \draw [-,>=latex] (H) -- (C)   node[midway,below right] {};
  \draw [-,>=latex] (H) -- (D)   node[midway,below right] {};
  \draw [-,>=latex] (H) -- (E)   node[midway,below right] {};
  \draw [-,>=latex] (H) -- (F)   node[midway,below right] {};
  \draw [-,>=latex] (A) -- (O)   node[midway,below left] {};
  \draw [-,>=latex] (A) -- (N)   node[midway,below left] {};
  \draw [-,>=latex] (A) -- (M)   node[midway,below left] {};
  \draw [-,>=latex] (A) -- (J)   node[midway,below left] {};
  \draw [-,>=latex] (A) -- (K)   node[midway,below left] {};
  \draw [-,>=latex] (A) -- (L)   node[midway,below left] {};  
  
  \draw [-,>=latex] (M) -- (P)   node[midway,below left] {};
  \draw [-,>=latex] (N) -- (P)   node[midway,below left] {};
  \draw [-,>=latex] (O) -- (P)   node[midway,below left] {};
  \draw [-,>=latex] (J) -- (P)   node[midway,below left] {};
  \draw [-,>=latex] (K) -- (P)   node[midway,below left] {};
  \draw [-,>=latex] (L) -- (P)   node[midway,below left] {};  
  
\end{scope}

\end{tikzpicture}

\end{center}


\section{Construction of intermediate extensions}

\subsection{\Large Absolut genus field}
In chapter $7$ of [\ref{Ishi}], Ishida has explicitly given the genus field of any pure field. For the pure quintic field $\Gamma = \mathbb{Q}(\sqrt[5]{n})$ where $n$ is $5^{th}$ power-free naturel number, we have the following theorem.
\begin{theorem}
Let $\Gamma \,=\, \mathbb{Q}(\sqrt[5]{n})$ be a pure quintic field, where $n \geq 2$ is $5^{th}$ power-free naturel number, and let $p_1\,,,,,p_r$ be all primes divisors of $n$ such that $p_i$ congruent 1 $(\mathrm{mod}\, 5)$ for each $i \in \{1,,,,,r\}$. Let $\Gamma^*$ be the absolute genus field of $\Gamma$, then
\begin{center}
$\Gamma^*\, = \, \prod_{i=1}^{r}M(p_i).\Gamma$ 
\end{center}
where $M(p_i)$ denotes the unique subfield of degree 5 of the cyclotomic field $\mathbb{Q}(\zeta_{p_i})$. The genus number of $\Gamma$ is given by, $g_{\Gamma}\, = \, [\Gamma^*:\Gamma]\, = \, 5^r$.
\end{theorem}
\begin{remark}
\item - If no prime $p \, \equiv \, 1\, (\mathrm{mod}\, 5)$ divides $n$, i.e $r=0$, then $\Gamma^*\, = \, \Gamma$.
\item - For any $r \geq 0$, $\Gamma^*$ is contained in the Hilbert $5$-class field $\Gamma_5^{(1)}$ of $\Gamma$.
\end{remark}
\begin{corollary}
\label{prem+}
\item(1) If $5$ divides exactly the class number $h_\Gamma$, then there is at most one prime $p \, \equiv \, 1\, (\mathrm{mod}\, 5)$ divides the integer $n$.
\item (2) If $5$ divides exactly $h_\Gamma$, and if a prime $p \, \equiv \, 1\, (\mathrm{mod}\, 5)$ divides $n$, then $\Gamma^{*} \, =\, \Gamma_5^{(1)}$; $(\Gamma^{'})^{*}\, =\, (\Gamma^{'})_5^{(1)}$; $(\Gamma^{''})^{*}\, =\, (\Gamma^{''})_5^{(1)}$; $(\Gamma^{'''})^{*}\, =\, (\Gamma^{'''})_5^{(1)}$; $(\Gamma^{''''})^{*}\, =\, (\Gamma^{''''})_5^{(1)}$. Furthermore, $k.\Gamma_5^{(1)}\,=\,k.(\Gamma^{'})_5^{(1)}\,=\,k.(\Gamma^{''})_5^{(1)}\,=\,k.(\Gamma^{'''})_5^{(1)}\,=\,k.(\Gamma^{''''})_5^{(1)}$
\end{corollary}
\begin{proof}
\item(1) If $p_1,,,,p_r$ are all prime congruent to $1\, (\mathrm{mod \, } 5)$ and divides $n$, then $5^r|h_\Gamma$ therfore if $5$ divides exactly $h_\Gamma$ then $r=1$, so one prime $p \, \equiv \, 1\, (\mathrm{mod}\, 5)$ divides $n$ exist, or no one.
\item(2) If $h_\Gamma$ is exactly divisible by $5$ and $p \, \equiv \, 1\, (\mathrm{mod}\, 5)$ divides $n$, then $g_\Gamma = 5$ so, $\Gamma^{*}\,=\, \Gamma_5^{(1)}\,=\, \Gamma.M(p)$, we have $g_\Gamma = 5$ so, $(\Gamma^{'})^{*}\,=\, \Gamma^{'}\prod_{i=1}^{r}.M(p_i)$ where $p_1,,,,p_r$ are prime congruent to $1$ $(\mathrm{mod}\, 5)$ and divides $n$, if $5$ divides exactly $h_\Gamma$ then $5$ divides exactly also $h_{\Gamma^{'}}$, because $\Gamma$ and $\Gamma^{'}$ are isomorphic, so we obtaint $(\Gamma^{'})^{*}\, = \, (\Gamma^{'})_5^{(1)}$. For the others equality we use the same reasoning. Moreover,
\begin{center}
\begin{equation}
\label{}
 \left\lbrace
   \begin{array}{ll}
   
   k.\Gamma_5^{(1)}\,=\,k.\Gamma.M(p)\,=\,k.M(p)\\ 
   k.(\Gamma^{'})_5^{(1)}\,=\,k.\Gamma^{'}.M(p)\,=\,k.M(p)\\
   k.(\Gamma^{''})_5^{(1)}\,=\,k.\Gamma^{''}.M(p)\,=\,k.M(p)\\
   k.(\Gamma^{'''})_5^{(1)}\,=\,k.\Gamma^{'''}.M(p)\,=\,k.M(p)\\
   k.(\Gamma^{''''})_5^{(1)}\,=\,k.\Gamma^{''''}.M(p)\,=\,k.M(p)\\
   
   \end{array}
   \right.
\end{equation}
\end{center}
Hence, $k.\Gamma_5^{(1)}\,=\,k.(\Gamma_5^{'})^{(1)}\,=\,k.(\Gamma_5^{''})^{(1)}\,=\,k.(\Gamma_5^{'''})^{(1)}\,=\,k.(\Gamma_5^{''''})^{(1)}$
\end{proof}

\begin{remark}\label{rem distin}
If $h_\Gamma$ is exactly divisible by $5$, and the radicand $n$ is not divisible by any prime $p \, \equiv \, 1\, (\mathrm{mod}\, 5)$, then the fields $k.\Gamma_5^{(1)}$, $k.(\Gamma_5^{'})^{(1)}$, $k.(\Gamma_5^{''})^{(1)}$, $k.(\Gamma_5^{'''})^{(1)}$ and $k.(\Gamma_5^{''''})^{(1)}$ are distincts.
\end{remark}

\subsection{\Large Relative genus field $(k/k_0)^*$ of $k$ over $k_0$}
Let $\Gamma\,=\,\mathbb{Q}(\sqrt[5]{n})$ be a pure quintic field, $k_0\,=\,\mathbb{Q}(\zeta_5)$ the $5^{th}$-cyclotomic field and $k\,=\,\Gamma(\zeta_5)$ be the normal closure of $\Gamma$. The relative genus field $(k/k_0)^*$ of $k$ over $k_0$ is the maximal abelian extension of $k_0$ which is contained in the Hilbert $5$-class field $k_5^{(1)}$ of $k$. A class $\mathcal{A} \in C_{k,5}$ is called ambiguous class relativly to $k/k_0$ if $\mathcal{A}^\sigma\,=\,\mathcal{A}$ such that $Gal(k/k_0)\,=\,\langle\sigma\rangle$. We note the group of ambiguous classes by $C_{k,5}^{(\sigma)}$.\\
The following proposition summarize some importants results of \ref{Mani}. Since $k$ is quintic Kummer extension of $k_0$, we can write $k\,=\,k_0(\sqrt[5]{n})$ such that $n\,=\,\mu\lambda^{e_\lambda}\pi_1^{e_1}....\pi_g^{e_g}$, where $\mu$ is unity of $\mathcal{O}_{k_0}$, $\lambda\,=\,1-\zeta_5$ the unique prime above $5$ in $k_0$, $\pi_i\,(1\leq i\leq g)$ are prime elements of $k_0$ such that $\pi_i\,\equiv\,a\,(\mathrm{mod}\,5\mathcal{O}_{k_0})$ with
$a\in \{1,2,3,4\}$ and $e_\lambda,e_i \in \{0,1,2,3,4\}$. We note by $d$ the number of ramified prime in $k/k_0$.
\begin{proposition}\label{prop genre}
\item[$(1)$] $C_{k}^{(\sigma)}\,=\,\{\mathcal{A}\in C_{k}\,|\, \mathcal{A}^\sigma = \mathcal{A}\}$ the subgroup of ambiguous classes coincide with the $5$-class group $C_{k,5}^{(\sigma)}\,=\,\{\mathcal{A}\in C_{k,5}\,|\,\mathcal{A}^\sigma = \mathcal{A}\}$.
\item[$(2)$] rank $C_{k,5}^{(\sigma)}\,=\,d-3+q^*$.
\item[$(3)$]$q^*$ = $ \begin{cases}
2 & \text{if }\, \zeta_5 , \zeta_5+1 \text{are norme of element in}\, k^*.\\
1 & \text{if some}\, \zeta_5^i(\zeta_5+1)^j \text{are norme of element in}\,k^*.\\
0 & \text{if none}\, \zeta_5^i(\zeta_5+1)^j \text{are norme of element in}\, k^*.\\
\end{cases}$
\item[(4)] Let $k\,=\,k_0(\sqrt[5]{n})$, writing $n$ = $\mu\lambda^{e_{\lambda}}\pi_{1}^{e_1}....\pi_{f}^{e_f}\pi_{f+1}^{e_{f+1}}.....\pi_{g}^{e_g}$, where each $\pi_i \,\equiv\, \pm1,\pm7\, (\mathrm{mod}\,\lambda^5)$ for $1\leq i \leq f$ and $\pi_j \,\not\equiv\, \pm1,\pm7 \, (\mathrm{mod}\,\lambda^5)$ for $f+1\leq j \leq g$. Then we have:
\begin{itemize}
\item[$(i)$] There exists $h_i \in \{1,..,4\}$ such that $\pi_{f+1}\pi_i^{h_i}\,\equiv\, \pm1,\pm7\, (\mathrm{mod}\,\lambda^5)$, for $f+2\leq i \leq g$.
\item[$(ii)$] if $n\,\not\equiv\,\pm1\pm7\, (\mathrm{mod}\,\lambda^5)$ and $q^*\,=\,1$, then the genus field $(k/k_0)^*$ is given as:
\begin{center}
$(k/k_0)^*$ = $k(\sqrt[5]{\pi_1},....\sqrt[5]{\pi_f},\sqrt[5]{\pi_{f+1}\pi_{f+2}^{h_{f+2}}},....\sqrt[5]{\pi_{f+1}\pi_{g}^{h_{g}}})$
\end{center}
with $h_i$ as in  (i).
\item[$(iii)$] In the other cases of $q^*$ and the congruence of $n$, the genus field  $(k/k_0)^*$ is given by deleting an appropriate number of $5^{th}$ root from the right side of (ii).
\end{itemize}

\end{proposition}
For more details see \ref{Mani}.\\
If $C_{k,5}$ is of type $(5,5)$ and rank $C_{k,5}^{(\sigma)}\,=\,1$, according to \ref{FOU1} we have three possible forms of the radicand $n$, and the genus field $(k/k_0)^*$ is given explicitly as follows
\begin{lemma}\label{lemma genus}
Let $p$ and $q$ prime numbers such that $p\,\equiv\,-1\,(\mathrm{mod}\,5)$ and $q\,\equiv\,\pm2\,(\mathrm{mod}\,5)$. Let $\pi_1$ and $\pi_2$ primes of $k_0$ such that $p\,=\,\pi_1\pi_2$. If $C_{k,5}$ is of type $(5,5)$ and rank $C_{k,5}^{(\sigma)}\,=\,1$, then we have:
\item[$(1)$] $n\,=\,5^ep\,\not\equiv\,\pm1\pm7\,(\mathrm{mod}\,25)$, where $e\in \{1,2,3,4\}$ and $p\,\not\equiv\,-1\,(\mathrm{mod}\,25)$. Then $(k/k_0)^*\,=\,k(\sqrt[5]{\lambda^\alpha\pi_1^{\alpha_1}\pi_2^{\alpha_2}})$, with $ \alpha,\alpha_1,\alpha_2 \in \{1,2,3,4\}$.
\item[$(2)$] $n\,=\,p^eq\,\equiv\,\pm1\pm7\,(\mathrm{mod}\,25)$, where $e\in \{1,2,3,4\}$, $p\,\not\equiv\,-1\,(\mathrm{mod}\,25)$ and $q\,\not\equiv\,\pm7\,(\mathrm{mod}\,25)$. Then $(k/k_0)^*\,=\,k(\sqrt[5]{q\pi_i^{\alpha_i}})$, with $i = 1,2$ and $\alpha \in \{1,2,3,4\}$.
\item[$(3)$] $n\,=\,p^e\,\equiv\,\pm1\pm7\,(\mathrm{mod}\,25)$, where $e\in \{1,2,3,4\}$, $p\,\equiv\,-1\,(\mathrm{mod}\,25)$. Then $(k/k_0)^*\,=\,k(\sqrt[5]{\pi_1^{\alpha_1}\pi_2^{\alpha_2}})$, with $\alpha_1,\alpha_2 \in \{1,2,3,4\}$.
\end{lemma}
\begin{proof}
For the three forms of $n$, we refer the reader to \ref{FOU1}. For the  relative genus field $(k/k_0)^*$ of each case, its follows from $(4)$ of proposition \ref{prop genre} and [\ref{Mani}, theorem 5.16].
\end{proof}

\subsection{intermediate extensions of
$k_5^{(1)}/k$}
Let $k\,=\,\mathbb{Q}(\sqrt[5]{n},\zeta_5)$ be the normal closure of the pure quintic field $\Gamma\,=\,\mathbb{Q}(\sqrt[5]{n})$, $C_{k,5}$ be the $5$-class group of $k$. When $C_{k,5}$ is of type $(5,5)$, it has $6$ subgroups of order $5$, denoted by $H_i$, $1\leq i \leq 6$. Let $K_i$ be the intermediate extension of
$k_5^{(1)}/k$, corresponding by class field theory to $H_i$. In \ref{FOU2} we proved that $C_{k,5}\,\cong\, C_{k,5}^+\times C_{k,5}^-$ such that $C_{k,5}^{+}\,=\,\{\mathcal{A}\in C_{k,5}\,|\,\mathcal{A}^{\tau^2} = \mathcal{A}\}$ and $C_{k,5}^{-}\,=\,\{\mathcal{X}\in C_{k,5}\,|\,\mathcal{X}^{\tau^2} = \mathcal{X}^{-1}\}$ with $Gal(k/\Gamma)\,=\,\langle\tau\rangle$. The following theorem allow us to determine the six intermediate extensions of $k_5^{(1)}/k$, using the action of $Gal(k/\mathbb{Q})$ on $C_{k,5}$.
\begin{theorem}\label{theo 6 ext}
\item[$(1)$] rank $C_{k,5}^{(\sigma)}\,\geq \,1$.
\item[$(2)$] If rank $C_{k,5}^{(\sigma)}\,= \,1$ then $C_{k,5}^{(\sigma)}\,=\,C_{k,5}^{+}\,=\,C_{k,5}^{1-\sigma}$ with $C_{k,5}^{1-\sigma}$ the principal genus.
\item[$(3)$] Let $\Gamma_5^{(1)}, (\Gamma')_5^{(1)}, (\Gamma^{''})_5^{(1)}, (\Gamma^{'''})_5^{(1)}$ and $(\Gamma^{''''})_5^{(1)}$ be respectivly the Hilbert $5$-class fields of $\Gamma^{'},\,\Gamma^{''},\,\Gamma^{'''}$ and $\Gamma^{''''}$. If rank $C_{k,5}^{(\sigma)}\,= \,1$, then $k\Gamma_5^{(1)}, k(\Gamma')_5^{(1)}, k(\Gamma^{''})_5^{(1)}, k(\Gamma^{'''})_5^{(1)}, k(\Gamma^{''''})_5^{(1)}$ and $(k/k_0)^*$ are the six intermediate extensions of $k_5^{(1)}/k$. Furthermore $\tau^2$ permutes the fields $k.\Gamma_5^{(1)}$, $k.(\Gamma_5^{'})^{(1)}$, $k.(\Gamma_5^{''})^{(1)}$, $k.(\Gamma_5^{'''})^{(1)}$ and $k.(\Gamma_5^{''''})^{(1)}$

\end{theorem}

\begin{proof}
\item[$(1)$] Let $\mathcal{A}$ be a class that generate $C_{k,5}^+$. According to [\ref{FOU2}, lemma 5.2], $C_{k,5}^+\, \cong\, C_{\Gamma,5}$, then $\mathcal{A}$ can be identified with a class that generate  $C_{\Gamma,5}$.
\item[-] If $\mathcal{A} \in C_{k,5}^{(\sigma)}$ then $C_{k,5}^{+}\subset C_{k,5}^{(\sigma)}$ and rank $C_{k,5}^{(\sigma)} \geq 1$.
\item[-] If $\mathcal{A} \notin C_{k,5}^{(\sigma)}$ then $\mathcal{A}^{\sigma}\neq \mathcal{A}$. Furthermore $\mathcal{A}^\sigma \neq \mathcal{A}^2$, otherwise $\mathcal{A}^\sigma = \mathcal{A}^2 \Rightarrow \mathcal{A}^{\sigma^2} = (\mathcal{A}^\sigma)^2 = \mathcal{A}^4 \Rightarrow \mathcal{A}^{\sigma^3} = (\mathcal{A}^\sigma)^4 = \mathcal{A}^8=\mathcal{A}^3 \Rightarrow \mathcal{A}^{\sigma^4} = (\mathcal{A}^\sigma)^3 = \mathcal{A}^6 = \mathcal{A} \Rightarrow \mathcal{A}^{\sigma^5} = \mathcal{A}^\sigma = \mathcal{A}$, which is impossible since $\mathcal{A}^\sigma \neq \mathcal{A}$. By the same reasoning we show that $\mathcal{A}^\sigma \neq \mathcal{A}^3$, $\mathcal{A}^\sigma \neq \mathcal{A}^4$. Ainsi $\mathcal{A}$ et $\mathcal{A}^\sigma$ generate $C_{k,5}$. Since $\mathcal{A}^5 = 1$ and $\mathcal{A}^{1+\sigma+\sigma^2+\sigma^3+\sigma^4} = 1$ then $\mathcal{A}^{\sigma^4} = \mathcal{A}^{4+4\sigma+4\sigma^2+4\sigma^3}$, this equality allow us to prove that $\mathcal{A}^{\sigma^3+3\sigma^2+2\sigma^2-1}$ is ambiguous class. Thus  rank $C_{k,5}^{(\sigma)} \geq 1$
\item[$(2)$] According to \ref{FOU1} or [\ref{Limura}, Lemma 3.3], since rank $C_{k,5}^{(\sigma)}\,=\,1$ then the radicand $n$ is not divisible by any prime $p$ such that $p\,\equiv\,1\,(\mathrm{mod}\, 5)$. By [\ref{FOU1}, Theorem 3.2] all elements of the group of strong ambiguous ideal classes, denoted $C_{k,s}^{(\sigma)}$ are fixed by $\tau^2$ and according to \ref{Gras} $C_{k,5}^{(\sigma)}/C_{k,s}^{(\sigma)}$ is generated by a class modulo $C_{k,s}^{(\sigma)}$ of element choosen in $(C_{k,5}^{(\sigma)})^+$. Thus all elements of $C_{k,5}^{(\sigma)}$ are fixed by $\tau$ and $\tau^2$, then $C_{k,5}^{(\sigma)} \subset C_{k,5}^{+}$. Therefor $C_{k,5}^{(\sigma)}\,=\,C_{k,5}^{+}$ because they are groups of order $5$. Let $\mathcal{A} \in C_{k,5}^+$, since $(\mathcal{A}^{1-\sigma})^{\tau^2}\,=\,\mathcal{A}^{\tau^2-\tau^2\sigma}\,=\,\mathcal{A}^{\tau^2-\sigma\tau^2}\,=\,(\mathcal{A}^{\tau^2})^{1-\sigma}\,=\,\mathcal{A}^{1-\sigma}$ then $(C_{k,5}^+)^{1-\sigma}\,\subset\, C_{k,5}^+$. By the same reasoning for $\mathcal{X} \in C_{k,5}^-$ we can prove that $(C_{k,5}^-)^{1-\sigma}\,\subset\, C_{k,5}^+$ because $\mathcal{X}$ and $\mathcal{X}^{-1}$ are in $C_{k,5}^-$. Since $C_{k,5}\,=\,\langle\mathcal{A},\mathcal{X}\rangle$ we get that $C_{k,5}^{1-\sigma}\,\subset\, C_{k,5}^+$. Thus $C_{k,5}^{(\sigma)}\,=\,C_{k,5}^{+}$ because they are groups of order $5$.
\item[$(3)$] Its sufficient to see that the six extensions are quintic cyclic and unramified extensions of $k$, then they are all contained in $k_5^{(1)}$. Precisly, since $Gal(k/\Gamma)\,=\, \langle\tau\rangle$, and by class field theory, $k\Gamma_5^{(1)}$ correspond to the subgroup $C_{k,5}^-$. Its easy to see that $C_{k,5}^-\,=\,C_{k,5}^{1-\tau^2}$. The relative genus field $(k/k_0)^*$ correspond to the principal genus $C_{k,5}^{1-\sigma}$ and by the previous point $(k/k_0)^*$ correspond to $C_{k,5}^{(\sigma)}$. We denote by $K_1\,=\,(k/k_0)^*$,  $K_2\,=\,k\Gamma_5^{(1)}$, $K_3\,=\,k(\Gamma')_5^{(1)}$, $K_4\,=\,k(\Gamma'')_5^{(1)}$, $K_5\,=\,k(\Gamma''')_5^{(1)}$ and $K_6\,=\,k(\Gamma'''')_5^{(1)}$. Since $K_2^{\tau^2}\,=\,K_4$, $K_4^{\tau^2}\,=\,K_6$, $K_6^{\tau^2}\,=\,K_3$, $K_3^{\tau^2}\,=\,K_5$ and $K_5^{\tau^2}\,=\,K_2$, then $\tau^2$ permutes the fields $K_2, K_3, K_4, K_5$ and $K_6$. According to [\ref{FOU2}, proposition 5.1], Since $C_{k,5}$ is of type $(5,5)$, we have $5$ divides exactly $h_\Gamma$, and by remark \ref{rem distin}, the six fields  $K_1, K_2, K_3, K_4, K_5$ and $K_6$ are distincts
\end{proof}

\section{Capitulation problem}
Let $k\,=\,\mathbb{Q}(\sqrt[5]{n},\zeta_5)$ be the normal closure of the pure quintic field $\Gamma\,=\,\mathbb{Q}(\sqrt[5]{n})$, $C_{k,5}$ be the $5$-class group of $k$. $C_{k,5}$ is elementary abelian bicyclic of type $(5,5)$. We begin with illuminating the generators of $C_{k,5}$ whenever rank $C_{k,5}^{(\sigma)}\,=\,1$.

\subsection{Generators of $C_{k,5}$}
According to \ref{FOU2}, if $C_{k,5}$ is of type $(5,5)$ and rank $C_{k,5}^{(\sigma)}\,=\,1$, we have an explicit study of generators of $C_{k,5}$, by means of the relation $C_{k,5}\,\cong\, C_{k,5}^+\times C_{k,5}^-$ and the three forms of the radicand $n$ proved in \ref{FOU1}, and already stated in the introduction. we have the following:
\begin{theorem}
\label{theo gener}
Let $k\,=\,\mathbb{Q}(\sqrt[5]{n},\zeta_5)$, where $n$ is a positive integer $5^{th}$ power-free, be the normal closure of the pure quintic field $\Gamma\,=\,\mathbb{Q}(\sqrt[5]{n})$ and $Gal(k/\Gamma)\,=\,\langle\tau\rangle$. Let $p$,$q$ and $l$ primes such that, $p\,\equiv\,-1\, (\mathrm{mod}\, 5)$, $q\,\equiv\,\pm2\, (\mathrm{mod}\, 5)$ and $l\,\neq\,p$, $l\,\neq\,q$ . Assume that $C_{k,5}$ is of type $(5,5)$ and rank$(C_{k,5}^{(\sigma)})\,=\,1$, then we have:
\begin{itemize}
\item[(1)] If $n\,=\,5^ep\,\not\equiv\,\pm1\pm7\,(\mathrm{mod}\,25)$, with $e\in\{1,2,3,4\}$ and $p\,\not\equiv\,-1\, (\mathrm{mod}\, 25)$. The prime $p$ decompose in $k$ as $p\mathcal{O}_k\,=\,\mathcal{P}^5_1\mathcal{P}^5_2$, where $\mathcal{P}_1$ and $\mathcal{P}_2$ are prime ideals of $k$. Let $\mathcal{L}$ a prime ideal of $k$ above $l$. If $5$ and $l$ are not a quintic residue modulo $p$, then the $5$-class group $C_{k,5}$ is generated by classes $[\mathcal{P}_1]$ and $[\mathcal{L}]^{1-\tau^2}$ and we have:
\begin{center}
$C_{k,5}$ = $\langle[\mathcal{P}_1]\rangle\times\langle[\mathcal{L}]^{1-\tau^2}\rangle$ = $\langle[\mathcal{P}_1],[\mathcal{L}]^{1-\tau^2}\rangle$ 
\end{center}

\item[(2)] If $n\,=\,p^eq\,\equiv\,\pm1\pm7\,(\mathrm{mod}\,25)$, with $e\in\{1,2,3,4\}$ and $p\,\not\equiv\,-1\, (\mathrm{mod}\, 25)$, $q\,\not\equiv\,\pm7\, (\mathrm{mod}\, 25)$. The prime $p$ decompose in $k$ as $p\mathcal{O}_k\,=\,\mathcal{P}^5_1\mathcal{P}^5_2$, where $\mathcal{P}_1$ and $\mathcal{P}_2$ are prime ideals of $k$. Let $\mathcal{L}$ a prime ideal of $k$ above $l$. If $q$ and $l$ are not a quintic residue modulo $p$, then the $5$-class group $C_{k,5}$ is generated by classes $[\mathcal{P}_1]$ and $[\mathcal{L}]^{1-\tau^2}$ and we have:
\begin{center}
$C_{k,5}$ = $\langle[\mathcal{P}_1]\rangle\times\langle[\mathcal{L}]^{1-\tau^2}\rangle$ = $\langle[\mathcal{P}_1],[\mathcal{L}]^{1-\tau^2}\rangle$ 
\end{center}

\item[(3)] If $n\,=\,p^e\,\equiv\,\pm1\pm7\,(\mathrm{mod}\,25)$, with $e\in\{1,2,3,4\}$ and $p\,\equiv\,-1\, (\mathrm{mod}\, 25)$. $5$ decompose in $k$ as $5\mathcal{O}_k\,=\,\mathcal{B}^4_1\mathcal{B}^4_2\mathcal{B}^4_3\mathcal{B}^4_4\mathcal{B}^4_5$, where $\mathcal{B}_i$  are prime ideals of $k$. If $5$ is not a quintic residue modulo $p$, then the $5$-class group $C_{k,5}$ is generated by classes $[\mathcal{B}_i]$ and $[\mathcal{B}_j]$, $i\,\neq\,j\in\{1,2,3,4,5\}$ and we have:
\begin{center}
$C_{k,5}$ = $\langle[\mathcal{B}_i]\rangle\times\langle[\mathcal{B}_j]\rangle$ = $\langle[\mathcal{B}_i],[\mathcal{B}_j]\rangle$ 
\end{center}
\end{itemize}
\end{theorem}

\subsection{Study of capitulation}
Let $k\,=\,\mathbb{Q}(\sqrt[5]{n},\zeta_5)$ be the normal closure of the pure quintic field $\Gamma\,=\,\mathbb{Q}(\sqrt[5]{n})$, $C_{k,5}$ be the $5$-class group of $k$. When $C_{k,5}$ is of type $(5,5)$, it has $6$ subgroups of order $5$, denoted by $H_i$, $1\leq i \leq 6$. Let $K_i$ be the intermediate extension of $k_5^{(1)}/k$, corresponding by class field theory to $H_i$. As each $K_i$ is cyclic of order $5$ over $k$, there is at least one subgroup of order $5$ of $C_{k,5}$, i.e.,
at least one $H_l$ for some $l \in \{1,2,3,4,5,6\}$, which capitulates in $K_i$ (by Hilbert's theorem $94$).
\begin{definition}
Let $\mathcal{S}_j$ be a generator of $H_j$ $(1\leq j \leq 6)$ corresponding to $K_j$. For $1\leq j \leq 6$, let $i_j \in  \{0,1,2,3,4,5,6\}$. We say that the capitulation is of type $(i_1,i_2,i_3,i_4,i_5,i_6)$ to mean the following:
\item[$(1)$] when $i_j \in \{1,2,3,4,5,6\}$, then only the class $\mathcal{S}_{i_j}$ and its powers capitulate in $K_j$;
\item[$(2)$] when $i_j\,=\,0$, then all the $5$-classes capitulate in $K_j$.
\end{definition}
By theorem \ref{theo 6 ext},we have that $C_{k,5}^{(\sigma)} \,=\,C_{k,5}^{+}$. Let $C_{k,5}^+\,=\,\langle\mathcal{A}\rangle$ and $C_{k,5}^-\,=\,\langle\mathcal{X}\rangle$. According to [\ref{FOU2}, Lemma 5.2], $C_{k,5}^+\,\cong\,C_{\Gamma,5}$, we can consider $C_{\Gamma,5}$ as subgroup of $C_{k,5}$, and we may choose the class $\mathcal{A}$ as $\mathcal{A}\,=\,[j_{k/\Gamma}(\mathcal{B})]$, when $\mathcal{B}$ is ideal of $\Gamma$ such that its class generates $C_{\Gamma,5}$.\\
We order the six intermediat fields $K_i$ of $k_5^{(1)}/k\, (1\leq i \leq 6)$ as follows
\begin{itemize}
\item[-] $K_1\,=\,(k/k_0)^*$ correspond to the subgroup  $H_1\,=\,C_{k,5}^{(\sigma)}\,=\,C_{k,5}^+\,=\,\langle\mathcal{A}\rangle$.
\item[-] $K_2\,=\,k\Gamma_5^{(1)}$ correspond to the subgroup  $H_2\,=\,C_{k,5}^{1-\tau^2}\,=\,C_{k,5}^-\,=\,\langle\mathcal{X}\rangle$.
\item[-] $K_3\,=\,k(\Gamma^{'})_5^{(1)}$, $K_4\,=\,k(\Gamma^{''})_5^{(1)}$, $K_5\,=\,k(\Gamma^{'''})_5^{(1)}$ and $K_6\,=\,k(\Gamma^{''''})_5^{(1)}$.
\end{itemize}
Our principal result on capitulation can be stated as follows:
\begin{theorem}\label{theo capit} Using the same notation as above.
\item[$(1)$] The class $\mathcal{A}$ capitulates in the six quintic extensions $K_i,\,i = 1,2,3,4,5,6.$
\item[$(2)$] The same number of classes capitulate in $K_2, K_3, K_4, K_5$ and $K_6$. More precisely, the possible capitulation types are $(0,0,0,0,0,0)$, $(1,0,0,0,0,0)$, $(0,1,1,1,1,1)$ or $(1,1,1,1,1,1)$.
\end{theorem}
\begin{proof}
\item[$(1)$] Let $C_{k,5}^{(\sigma)}\,=\,C_{k,5}^+\,=\,\langle\mathcal{A}\rangle$, where $\mathcal{A}\,=\,[j_{k/\Gamma}(\mathcal{B})]$ such that $\mathcal{B}$ is ideal of $\Gamma$. Since $\Gamma_5^{(1)}$ is Hilbert $5$-class field of $\Gamma$, then $\mathcal{B}$ becames pricipal in $\Gamma_5^{(1)}$. Thus when $\mathcal{B}$ considred as ideal of  $k\Gamma_5^{(1)}$, $\mathcal{B}$ becames principal in  $k\Gamma_5^{(1)}$. Then we get that $\mathcal{A}$ capitulates in  $k\Gamma_5^{(1)}\,=\,K_2$. As $K_4\,=\,K_2^{\tau^2}$, $K_6\,=\,K_4^{\tau^2}$, $K_3\,=\,K_6^{\tau^2}$, $K_5\,=\,K_3^{\tau^2}$ and $\mathcal{A}^{\tau^2}\,=\,\mathcal{A}$, we deduce that $\mathcal{A}$ capitulates in $K_2, K_3, K_4, K_5$ and $K_6$. To prove that $\mathcal{A}$ capitulates in $K_1$, we use Tanaka-Therada's theorem. Since $K_1\,=\,(k/k_0)^*$ is the genus field of $k/k_0$, and $\mathcal{A}$ is ambiguous class, then $\mathcal{A}$ capitulates in $K_1$.

\item[$(2)$] We have $\tau^2$ permute the fields $K_6,K_5, K_4, K_3, $ and $K_2$, and $C_{k,5}^{\tau^2}\,=\,C_{k,5}$. Since each element of $C_{k,5}^{(\sigma)}\,=\,\langle\mathcal{A}\rangle$ capitulates in $K_6, K_5, K_4, K_3, $ and $K_2$, then we have two possibilities: Only $\mathcal{A}$ and its powers capitulate in the extensions $K_6, K_5, K_4, K_3, $ and $K_2$, or all the $5$-classes capitulate in $K_6, K_5, K_4, K_3, $ and $K_2$. This give us as type of capitulation, $(i_1,0,0,0,0,0)$ or $(i_1,1,1,1,1,1)$ with $i_1 \in \{0,1,2,3,4,5,6\}$. Since $\mathcal{A}$ capitulates in $K_1$, then $i_1\,\neq\,0$, so $i_1$ is necessary $1$. Thus we prove that the types of capitulation possibles are $(0,0,0,0,0,0)$, $(0,1,1,1,1,1)$, $(1,0,0,0,0,0)$ or $(1,1,1,1,1,1)$.
\end{proof}


\end{document}